\documentclass[12pt]{amsart}

\usepackage{amsmath}
\usepackage{amssymb}
\usepackage{amsthm}
\usepackage{amscd}
\usepackage{xypic}

\swapnumbers \textwidth=16.00cm \textheight=22cm \topmargin=0.00cm
\headsep=1cm \numberwithin{equation}{section}
\def\P{{\mathbb P}}

\def\Z{{\mathbb Z}}

\newtheorem{theorem}{Theorem}[section]
\newtheorem{lemma}[theorem]{Lemma}
\newtheorem{proposition}[theorem]{Proposition}
\newtheorem{corollary}[theorem]{Corollary}

\theoremstyle{definition}
\newtheorem{definition}[theorem]{Definition}
\newtheorem{definition and remark}[theorem]{Definition and Remark}
\newtheorem{remark}[theorem]{Remark}
\newtheorem{remark and definition}[theorem]{Remark and Definition}
\newtheorem{remark and notation}[theorem]{Remark and Notation}
\newtheorem{notation and remark}[theorem]{Notation and Remark}
\newtheorem{notation and convention}[theorem]{Notation and Convention}

\newtheorem{notation and remarks}[theorem]{Notation and Remarks}
\newtheorem{notation and reminder}[theorem]{Notation and Reminder}

\newtheorem{construction and examples}[theorem]{Construction and Examples}

\newtheorem{problem and remark}[theorem]{Problem and Remark}

\newcommand\Hom{\operatorname{Hom}}

\title{Some remarks on the $\mathcal{K}_{p,1}$ Theorem}

\begin{document}

\author{Yeongrak Kim}
\address{Department of Mathematics, Pusan National University, Busan 46241, Republic of Korea,}
\address{Institute of Mathematical Science, Pusan National University, Busan 46241, Republic of Korea}
\email{yeongrak.kim@pusan.ac.kr}

\author{Hyunsuk Moon}
\address{Department of Mathematics, Korea Institute for Advanced Study, Seoul 02455, Republic of Korea}
\
\email[Correspoinding Author]{hsmoon87@kias.re.kr}

\author{Euisung Park}
\address{Department of Mathematics, Korea University, Seoul 02841, Republic of Korea}
\email{euisungpark@korea.ac.kr}

\date{Seoul, \today}



\thispagestyle{empty} \maketitle

\setcounter{page}{1}

\begin{abstract}
Let $X$ be a non-degenerate projective irreducible variety of dimension $n \ge 1$, degree $d$, and codimension $e \ge 2$ over an algebraically closed field $\mathbb{K}$ of characteristic $0$. Let $\beta_{p,q} (X)$ be the $(p,q)$-th graded Betti number of $X$. M. Green proved the celebrating $\mathcal K_{p,1}$-theorem about the vanishing of $\beta_{p,1} (X)$ for high values for $p$ and potential examples of nonvanishing graded Betti numbers. Later, Nagel-Pitteloud and Brodmann-Schenzel classified varieties with nonvanishing $\beta_{e-1,1}(X)$. It is clear that $\beta_{e-1,1}(X) \neq 0$ when there is an $(n+1)$-dimensional variety of minimal degree containing $X$, however, this is not always the case as seen in the example of the triple Veronese surface in $\mathbb{P}^9$. 

In this paper, we completely classify varieties $X$ with nonvanishing $\beta_{e-1,1}(X) \neq 0$ such that $X$ does not lie on an $(n+1)$-dimensional variety of minimal degree. They are exactly cones over smooth del Pezzo varieties whose Picard number is $\le n-1$.
\end{abstract}

\setcounter{page}{1}

\section{Introduction}\label{sect:Introduction}
\noindent It is a fundamental challenge in algebraic geometry to understand the interplay between the geometric properties of a projective variety embedded in a projective space and the algebraic properties such as the graded Betti numbers of its homogeneous ideal. Modern research in this direction began with the pioneering paper \cite{Gr2} by M. Green. To be precise, let $X \subset \P^{n+e}$ be a non-degenerate projective irreducible variety of dimension $n \ge 1$, codimension $e \geq 2$ and degree $d$ over an algebraically closed field $\mathbb{K}$ of characteristic zero. Let $R$ be the homogeneous coordinate rings of $\P^{n+e}$. Also let $I(X)$ and $R(X) := R/I(X)$ be respectively the homogeneous ideal and the homogeneous coordinate ring of $X$ in $\P^{n+e}$. For each $p,q \ge 1$, the $(p,q)$-th graded Betti number of $X$ is defined by
\begin{equation*}
\beta_{p,q} (X) = \dim_{\mathbb{K}} ~ \mathcal{K}_{p,q} (R(X),R_1 )
\end{equation*}
where $\mathcal{K}_{p,q} (R(X),R_1 )$ is the $(p,q)$th Koszul cohomology group of $R(X)$ as a graded $R$-module.

In \cite{Gr2}, M. Green proves the following fundamental result.

\begin{theorem}[M. Green's $\mathcal{K}_{p,1}$ Theorem]\label{thm:Kp1}
Let $X \subset \P^{n+e}$ be as above. Then

\renewcommand{\descriptionlabel}[1]%
             {\hspace{\labelsep}\textrm{#1}}
\begin{description}
\setlength{\labelwidth}{13mm}
\setlength{\labelsep}{1.5mm}
\setlength{\itemindent}{0mm}
\item[{\rm (1)}] $\beta_{p,1} (X) = 0$ for $p > e$.

\item[{\rm (2)}] $\beta_{e,1} (X) > 0$ if and only if $d = e+1$.

\item[{\rm (3)}] $\beta_{e-1,1} (X) > 0$ if and only if either $d \leq e+2$ and $X$ is not a cone over the isomorphic projection of the Veronese surface in $\P^5$ or else $d \geq e+3$ and $X$ lies on an $(n+1)$-fold of minimal degree.
\end{description}
\end{theorem}

In Theorem \ref{thm:Kp1}, the first and second statements are Mark Green's original theorems, while the third one is obtained by combining his original theorem with others' later results. More precisely, recall that $d \geq e+1$, and if $d=e+1$ then $X$ is called \textit{a variety of minimal degree}. A variety is of minimal degree if and only if it is either (a cone over) the Veronese surface in $\P^5$ or else (a cone over) a smooth rational normal scroll. A modern proof of this classification can be found in \cite{EH} and \cite{F}. If $X$ is a variety of minimal degree, then $\beta_{p,q} (X) = 0$ for all $q \geq 2$ and
\begin{equation}\label{eq:VMD}
\beta_{p,1} (X) =  p {{e+1} \choose {p+1}} \quad \mbox{for all $p \geq 1$}.
\end{equation}
In particular, $\beta_{p,1} (X) = 0$ for all $p > e$, $\beta_{e,1} (X) = e$ and $\beta_{e-1,1} (X) = e^2 -1$. Regarding Theorem \ref{thm:Kp1}.$(3)$, U. Nagel and Y. Pitteloud in \cite[Theorem 3.5]{NP} proved that if $d \geq e+3$ and $X$ lies on an $(n+1)$-fold of minimal degree then
\begin{equation}\label{eq:d geq e+3}
\beta_{e-1,1} (X) = e-1.
\end{equation}
 Also, it is a well-known fact that if $d=e+2$ and $X$ is arithmetically Cohen-Macaulay then
\begin{equation}\label{eq:del Pezzo}
\beta_{p,1} (X) =  p {{e+1} \choose {p+1}} - {{e} \choose {p-1}} \quad \mbox{for all $1 \leq p \leq e-1$}
\end{equation}
(cf. \cite[Theorem 1]{Hoa}). In particular, $\beta_{e-1,1} (X) = {{e+1} \choose {2}} -1$ is a positive integer. When $d=e+2$ and $X$ is not arithmetically Cohen-Macaulay, the situation related to the above became clear later by the following result of M. Brodmann and P. Schenzel.

\begin{theorem}[Brodmann-Schenzel, \cite{BS1,BS2}]\label{thm:BS}
Let $X \subset \P^{n+e}$ be as above. Suppose that $d = e+2$ and $X$ is not arithmetically Cohen-Macaulay and not a cone over the isomorphic projection of the Veronese surface in $\P^5$. Then $X$ lies on an $(n+1)$-fold rational normal scroll and $\beta_{e-1,1} (X) = e-1$.
\end{theorem}

Going back to Theorem \ref{thm:Kp1}.$(3)$, suppose that there exists an $(n+1)$-fold of minimal degree $Y \subset \P^{n+e}$ which contains $X$. Then we have
\begin{equation*}
\mathcal{K}_{e-1,1} (R(Y),R_1 )  \subseteq \mathcal{K}_{e-1,1} (R(X),R_1 )
\end{equation*}
and hence
\begin{equation*}
e-1 = \beta_{e-1,1} (Y) \leq \beta_{e-1,1} (X).
\end{equation*}
That is, the existence of such a variety $Y$ implies immediately that $\beta_{e-1,1} (X)$ is nonzero. Along this line, we define $\mu (X)$ to be the number of $(n+1)$-folds of minimal degree which contains $X$. In \cite[p. 151]{Gr2}, M. Green remarked that the smooth del Pezzo surface $X = \nu_3 (\P^2 ) \subset \P^9$ does not lie on a threefold of minimal degree
and thus $\mu (X) = 0$.

Along this line, our first result is the classification of all $X \subset \P^{n+e}$ such that $\beta_{e-1,1} (X)$ is nonzero and $\mu (X)$ is zero.

\begin{theorem}\label{thm:main1}
Let $X \subset \P^{n+e}$ be a non-degenerate projective irreducible variety of dimension $n \ge 1$, codimension $e \geq 2$ and degree $d$. Then the following conditions are equivalent:

\renewcommand{\descriptionlabel}[1]%
             {\hspace{\labelsep}\textrm{#1}}
\begin{description}
\setlength{\labelwidth}{13mm} \setlength{\labelsep}{1.5mm}
\setlength{\itemindent}{0mm}
\item[{$(i)$}] $\beta_{e-1,1} (X) \neq 0$ and $\mu (X) = 0$;

\item[{$(ii)$}] $X$ is a cone over $X_0$ where $X_0 \subset \P^r$ is either
\smallskip

\begin{enumerate}
\item[$(ii.1)$] $(d=9)$ $\nu_3 (\P^2 ) \subset \P^9$, or
\smallskip

\item[$(ii.2)$] $(d=8)$ $\nu_2 (\P^3 ) \subset \P^9$, or
\smallskip

\item[$(ii.3)$] $(d=7)$ the inner projection of $\nu_2 (\P^3 )$ into $\P^8$, or
\smallskip

\item[$(ii.4)$] $(d=6)$ $\P^2 \times \P^2 \subset \P^8$ or its smooth hyperplane section, or
\smallskip

\item[$(ii.5)$] $(d=5)$ a smooth quintic del Pezzo variety of dimension $\geq 3$ (i.e., a $k$-dimensional smooth linear section of $\mathbb{G} (1,4) \subset \P^9$ for some $3 \leq k \leq 6$).

\end{enumerate}
\item[{$(iii)$}]  $X$ is a cone over a smooth del Pezzo variety whose Picard number is $\le n-1$.

\end{description}
\end{theorem}
We give a proof of this result at the end of Section \ref{sect:MainThmProof}.

Let $X \subset \P^{n+e}$ be as above such that $e \geq 3$ and $d \geq e+2$ and $X$ is not a cone. By Theorem \ref{thm:Kp1}-$(3)$ and Theorem \ref{thm:main1}, it holds that $\mu (X) > 0$ if and only if either\\

\begin{enumerate}
\item[$(\alpha)$] $d = e+2$ and $X$ is not in the list of Theorem \ref{thm:main1}; or
\item[$(\beta)$] $d \geq e+3$ and $\beta_{e-1,1} (X) \neq 0$.\\
\end{enumerate}

\noindent Also, the case of $d=e+2$ is further divided into the following subcases:  \\

\begin{enumerate}
\item[$(\alpha.1)$] $X$ is not arithmetically Cohen-Mcaulay; or
\item[$(\alpha.2)$] $X$ is a linearly normal curve of arithmetic genus one; or
\item[$(\alpha.3)$] $X$ is a non-normal del Pezzo variety; or
\item[$(\alpha.4)$] $X \subset \P^d$ is a smooth del Pezzo surface of degree $d \in \{ 5, 6, 7, 8 \}$; or
\item[$(\alpha.5)$] $X = \P^1 \times \P^1 \times \P^1  \subset \P^3$; or
\item[$(\alpha.6)$] $X$ is a singular normal del Pezzo variety.\\
\end{enumerate}

\noindent In this paper we find the precise value of $\mu (X)$ for the cases of $(\alpha.1) \sim (\alpha.5)$ and $(\beta)$. For details, see Proposition \ref{prop:first computation of mu(X)}, Remark \ref{rmk:mu(X) non-normal del Pezzo}, Proposition \ref{prop:del Pezzo 3fold Segre}, Proposition \ref{prop:embedding scroll del Pezzo surfaces} and Theorem \ref{thm:uniqueness and syzygy scheme}.\\

There are two key ingredients for this paper: the first one is algebra and geometry of rational scrolls, and the other one is the syzygy scheme of a nonzero Koszul cohomology class $\gamma \in \mathcal K_{e-1,1} (X)$. Roughly speaking, the syzygy scheme $Syz(\gamma)$ is defined by quadrics which minimally induce the syzygy $\gamma$. It is particularly useful to study varieties $X$ with $\beta_{e-1,1}(X) \neq 0$ and $\mu(X) = 0, 1$. Indeed, we give a simple proof that $\mu(X)=1$ for the case $(\beta)$ and $(\alpha.1)$ using the syzygy scheme (Theorem \ref{thm:uniqueness and syzygy scheme}). We also compute sygyzy schemes for smooth del Pezzo varieties (Proposition \ref{prop:SyzGen}, \ref{prop:SyzDelPezzo}). In some cases, such computations could explain the reason how we could derive $\mu(X)$ from given $X$, which is originated from an idea of Eisenbud to find a determinantal variety \cite{Eis88}. 

In the Section \ref{sect:Grassmannian}, we concentrate on smooth $k$-dimensional linear sections of the Grassmannian $X_k = \mathbb{G} (1, \mathbb{P}^4) \cap \Lambda^{k+3} \subset \mathbb{P}^{k+3}$. It is clear that if $\mu(X_k)>0$, then an $(k+1)$-dimensional variety $Y$ of minimal degree containing $X_k$ is defined by quadrics of rank $\le 4$. We compute how many independent quadrics of given ranks are contained in each $I(X_k)$ which gives an alternative proof that $\mu(X_k)=0$ for $k \ge 3$, see Theorem \ref{thm:quintic del Pezzo manifolds} and Proposition \ref{prop:quintic del Pezzo, rank}.




\section{Varieties of almost minimal degree}\label{sect:VAMD}
\noindent Let $X \subset \P^{n+e}$ be a non-degenerate projective irreducible variety of dimension $n \ge 1$ and codimension $e \geq 2$. Due to \cite{BS2}, we say that $X$ is \textit{a variety of almost minimal degree} if $d=e+2$. This section aims to explain some properties of varieties of almost minimal degree.

We begin with the following structure theorem of varieties of almost minimal degree which are not arithmetically Cohen-Macaulay.

\begin{theorem}[M. Brodmann and P. Schenzel]\label{thm:Brod-Schenzel VAMD}
Let $X \subset \P^{n+e}$ be an $n$-dimensional variety of almost minimal degree which is not arithmetically Cohen-Macaulay. Then

\renewcommand{\descriptionlabel}[1]%
             {\hspace{\labelsep}\textrm{#1}}
\begin{description}
\setlength{\labelwidth}{13mm} \setlength{\labelsep}{1.5mm}
\setlength{\itemindent}{0mm}
\item[{\rm (1)}] $($\cite[Theorem 5.6]{BS2}$)$ $X$ is the projection of a variety $\tilde{X} \subset \P^{n+e+1}$ of minimal degree.

\item[{\rm (2)}] $($\cite[Theorem 2.1]{BS1}$)$ If $e=2$, then $\beta_{e-1,1} (X) =0$ if and only if $X$ is a cone over the isomorphic projection of the Veronese surface in $\P^5$.

\item[{\rm (3)}] $($\cite[Theorem 7.5 and 8.3]{BS2}$)$ Suppose that $e \geq 3$. Then $X$ is contained in an $(n+1)$-fold rational normal scroll $Y$ and hence $\mu (X) \geq 1$. Also $\beta_{e-1,1} (X) = \beta_{e-1,1}(Y) = e-1$.
\end{description}
\end{theorem}

\begin{proposition}\label{prop:first computation of mu(X)}
Let $X \subset \P^{n+e}$ be an $n$-dimensional variety of almost minimal degree with $e\geq 3$ which is not a cone. Then

\renewcommand{\descriptionlabel}[1]%
             {\hspace{\labelsep}\textrm{#1}}
\begin{description}
\setlength{\labelwidth}{13mm} \setlength{\labelsep}{1.5mm}
\setlength{\itemindent}{0mm}
\item[{\rm (1)}] If $X$ is not arithmetically Cohen-Macaulay, then $\mu (X) \geq 1$.
\smallskip

\item[{\rm (2)}] If $X$ is arithmetically Cohen-Macaulay and singular, then $\mu (X) \geq 1$.
\smallskip

\item[{\rm (3)}] If $X$ is a curve of arithmetic genus one, then $\mu (X) = \infty$.
\end{description}
\end{proposition}

\begin{proof}
$(1)$ By the Theorem \ref{thm:Brod-Schenzel VAMD}-(3).
\smallskip

\noindent $(2)$ Let $p$ be a singular point of $X$. Then $Y := {\mbox Join}(p,X)$ is a variety of minimal degree (cf. \cite[Lemma 3]{Sw}). In particular, $\mu (X) \geq 1$.

\noindent $(3)$ Let $p$ be a smooth point of $X$. Then for every smooth point $q$ of $X-\{p\}$, the line bundle $\mathcal{O}_X (p+q)$ defines a double covering $X \twoheadrightarrow \P^1$ and hence
it gives a rational normal surface scroll $S_q \subset \P^{1+e}$ such that $X$ is contained in $S_q$ and the line $\langle p,q \rangle$ is a ruling of $S_q$ (cf. \cite[Proposition 6.19]{Eis05}). Therefore $\mu (X) = \infty$.
\end{proof}

\begin{remark}\label{rmk:mu(X) non-normal del Pezzo}
Let $X \subset \P^{n+e}$ be as in Proposition \ref{prop:first computation of mu(X)}.
\smallskip

\noindent $(1)$ Regarding Proposition \ref{prop:first computation of mu(X)}-$(1)$, we will prove in section 5 that if $X$ is not arithmetically Cohen-Macaulay then $\mu (X) = 1$ and so there exists exactly one $(n+1)$-dimensional variety of minimal degree which contains $X$. For details, see Theorem \ref{thm:uniqueness and syzygy scheme}.
\smallskip

\noindent $(2)$ Regarding Proposition \ref{prop:first computation of mu(X)}-$(2)$, suppose that $X$ is a non-normal del Pezzo variety of dimension $n \geq 2$ and codimension $e \geq 3$. Then it holds that $n \leq 3$ and $X$ is the projection of a smooth rational normal scroll $\tilde{X} \subset \P^{n+e+1}$ such that if $n=2$ then $\tilde{X}$ is either $S(1,e+1)$ or $S(2,e)$, and if $n=3$ then $\tilde{X}$ is $S(1,1,e)$. For details, we refer the reader to \cite[Theorem 6.2]{BP1}. Then
\begin{equation*}
\mu (X) = \begin{cases}
1 & \mbox{if $\tilde{X}=S(1,e+1)$ or $S(1,1,e)$, and} \\
2 & \mbox{if $\tilde{X}=S(2,e)$}
\end{cases}
\end{equation*}
by \cite[Theorem 5.8 and 6.10]{BP2}.
\end{remark}

The following proposition is very useful for solving the problem of classifying all varieties of minimal degree which contain a given del Pezzo variety as a divisor.

\begin{proposition}\label{prop:smoothness of embedding scroll}
Let $X \subset \P^{n+e}$ be an $n$-dimensional variety of almost minimal degree. If $X$ is contained in an $(n+1)$-dimensional variety $Y \subset \P^{n+e}$ of minimal degree, then ${\rm Vert} (Y) \subseteq {\rm Sing} (X)$. In particular, if $X$ is smooth then so is $Y$.
\end{proposition}

\begin{proof}
Let $Y$ be a cone so that we may choose $y \in {\rm Vert} (Y)$. Then $Y = {\rm Join}(y,X)$. If $y \notin X$, then the linear projection of $X$ from $y$ is
a birational morphism from $X$ onto a general hyperplane section of $Y$. Thus
\begin{equation*}
\deg(Y) = \deg(X) =e+2
\end{equation*}
which is a contradiction. If $y \in X - {\rm Sing}(X)$, then
\begin{equation*}
\deg(Y) = \deg(X)-1 = e+1
\end{equation*}
which is again a contradiction. Therefore, we must have ${\rm Vert} (Y) \subseteq {\rm Sing} (X)$.
\end{proof}

We finish this section with mentioning the following well-known classification result of smooth del Pezzo varieties.

\begin{theorem}\label{thm:del Pezzo}
Let $X \subset \P^{n+e}$ be an $n$-dimensional smooth del Pezzo variety with $e \geq 3$. Then $X$ is of one of the following types.

\renewcommand{\descriptionlabel}[1]%
             {\hspace{\labelsep}\textrm{#1}}
\begin{description}
\setlength{\labelwidth}{13mm} \setlength{\labelsep}{1.5mm}
\setlength{\itemindent}{0mm}
\item[{\rm (1)}] $(n=1)$ $X \subset \P^{1+e}$ is an elliptic normal curve.

\item[{\rm (2)}] $($P. del Pezzo, \cite{DP}$)$ $(n=2)$ Either $X \subset \P^{9-t}$ is obtained by the inner projections of the triple Veronese surface $\nu_3 (\P^2 ) \subset \P^9$ from $t$ points in linearly general position for some $t \in \{0,1,2,3,4 \}$ or else $X \subset \P^8$ is the $2$-uple Veronese embedding of the smooth quadric surface $Q$ in $\P^3$.

\item[$(3)$] $($T. Fujita, \cite[Theorem (8.11)]{F}$)$ $(n \geq 3)$
\begin{enumerate}
\item[$(3.1)$] $(d=8)$ $X$ is the second Veronese variety $\nu_2 (\P^3 ) \subset \P^9$;
\item[$(3.2)$] $(d=7)$ $X$ is the inner projection of $\nu_2 (\P^3 ) \subset \P^9$;
\item[$(3.3)$] $(d=6)$ $X = \P^2 \times \P^2 \subset \P^8$ or $\P_{\P^2} (\mathcal{T}_{\P^2}) \subset \P^7$ or $\P^1 \times \P^1 \times \P^1 \subset \P^7$;
\item[$(3.4)$] $(d=5)$ $X$ is a linear section of $\mathbb{G} (1,\mathbb{P}^4) \subset \P^9$.\\
\end{enumerate}
\end{description}
\end{theorem}

\section{Varieties of minimal degree containing smooth del Pezzo varieties}\label{sect:smoothDP}
\noindent Let $X \subset \P^{n+e}$ be an $n$-dimensional del Pezzo variety with $e\geq 3$ which is not a cone. If $n=1$ or $X$ is singular, then $\mu(X)>0$ by Proposition \ref{prop:first computation of mu(X)}. On the other hand, if $X$ is smooth then $\mu(X)$ can be zero, as in the case of the third Veronese surface (cf. \cite[Page 151]{Gr2}). Along this line, the main purpose of this section is to give a precise description of all varieties of minimal degree which contains $X$ when $X$ is either the Segre variety $\P^1 \times \P^1 \times \P^1 \subset \P^7$ or else the smooth del Pezzo surface of degree $d \in \{5,6,7,8 \}$.

\begin{lemma}\label{lem:divisor class of del Pezzo manifolds}
Let $X \subset \P^{n+e}$ be an $n$-dimensional smooth del Pezzo variety with $e \geq 3$. Suppose that $Y \subset \P^{n+e}$ be an $(n+1)$-fold of minimal degree which contains $X$. Then

\renewcommand{\descriptionlabel}[1]%
             {\hspace{\labelsep}\textrm{#1}}
\begin{description}
\setlength{\labelwidth}{13mm} \setlength{\labelsep}{1.5mm}
\setlength{\itemindent}{0mm}
\item[{\rm (1)}] $Y$ is a smooth rational normal scroll.

\item[{\rm (2)}] $X \equiv 2H+(2-e)F$ as a divisor of $Y$ where $H$ and $F$ be respectively the hyperplane section of $Y$ in $\P^{n+e}$ and a fiber of the projection map $\pi : Y \rightarrow \P^1$.
\end{description}
\end{lemma}

\begin{proof}
$(1)$ Proposition \ref{prop:smoothness of embedding scroll} says that $Y$ is smooth. In particular, $Y$ cannot be a cone over the Veronese surface in $\P^5$. In consequence, $Y$ is a smooth rational normal scroll.

\noindent $(2)$ Since $\operatorname{Pic}(Y) = \Z H \oplus \Z F$, we can write $X \equiv aH+bF$ for some $a\geq 1$ and $b \in \mathbb{Z}$. Since $\deg (X) = e+2$ and $I(X)$ is generated by quadrics, we have either
\begin{equation*}
X \equiv H+2F \quad \mbox{or else} \quad X \equiv 2H+(2-e)F .
\end{equation*}
Also, if $X$ is linearly equivalent to $H+2F$ then it fails to be linearly normal (cf. \cite[Theorem 1.1]{P2007}). Therefore $X$ must be linearly equivalent to $2H+(2-e)F$.
\end{proof}

\begin{proposition}\label{prop:del Pezzo 3fold Segre}
Let $X = \P^1 \times \P^1 \times \P^1 \subset \P^7$. Then $\mu (X) = 3$. Moreover, if $Y \subset \P^7$ is a fourfold of minimal degree which contains $X$ then $Y \cong S(1,1,1,1)$.
\end{proposition}

\begin{proof}
Let $Y \subset \P^7$ be a fourfold of minimal degree which contains $X$. Then $Y$ is a smooth rational normal fourfold scroll and hence $Y \cong S(1,1,1,1)$. Let $\pi : Y \to \P^1$ be a projective bundle map. Then we have
\begin{equation*}
Y \cong S(1,1,1,1) = \bigcup_{\lambda \in \P^1} \P_{\lambda}^3 \quad \mbox{where $\P_{\lambda}^3 = \pi^{-1} (\lambda)$.}
\end{equation*}
Also, $X \equiv 2H - 2F$ by Lemma \ref{lem:divisor class of del Pezzo manifolds} and so $X \cap \P_{\lambda}^3$ is a quadric surface in $\P_{\lambda}^3$. Since the image of $\pi |_X : X \hookrightarrow Y \twoheadrightarrow \P^1$ is of dimension one, the pull-back $(\pi |_X )^{\ast} \mathcal O_{\P^1} (1)$ must be either $\mathcal O_X(1,0,0), \mathcal O_X (0,1,0)$ or $\mathcal O_X (0,0,1)$. This implies that
\begin{equation*}
X \cap \mathbb{P}_{\lambda}^3 = (\pi |_X)^{-1} (\lambda) = \{ \lambda \} \times \P^1 \times \P^1
\end{equation*}
in the first case, and so on. In particular, $\mathbb{P}_{\lambda}^3$ is the linear span of $\{ \lambda \} \times \P^1 \times \P^1$, and hence $Y$ is uniquely determined by the projection map $\pi : \P^1 \times \P^1 \times \P^1 \to \P^1$. There are exactly three such projection maps.
\end{proof}

Now let $S \subset \P^{2+e}$ be a smooth del Pezzo surface of degree $d = e+2 \ge 5$. Lemma \ref{lem:divisor class of del Pezzo manifolds} says that if there is a $3$-fold of minimal degree $Y$ in $\P^{2+e}$ which contains $X$, then $Y$ is a smooth rational normal scroll and $S$ is linearly equivalent to $2H + (4-d)F$ as a divisor of $Y$. Note also that $Y$ is defined by the $2$-minors of the matrix of the multiplication map
\[
\tau : H^0 (Y, \mathcal O_Y (F)) \otimes H^0 (Y, \mathcal O_Y (H-F)) \to   H^0 (\P^{2+e}, \mathcal O_{\P^{2+e}}(1)).
\]
More precisely, let
$$\varphi : H^0 (Y, \mathcal O_Y(H)) \rightarrow H^0 (\P^{2+e}, \mathcal O_{\P^{2+e}}(1))$$
be the isomorphism corresponding to the embedding of $Y$ into $\P^{2+e}$. If $\{ u_1 , u_2 \}$ and $\{ v_1 , \ldots , v_e \}$ are respectively bases for $H^0 (Y, \mathcal O_Y (F))$ and $H^0 (Y, \mathcal O_Y (H-F))$, then the matrix
\[
M(\tau) := \begin{pmatrix}
\varphi (u_1 \otimes v_1 ) & \varphi (u_1 \otimes v_2 ) & \cdots & \varphi (u_1 \otimes v_e ) \\
\varphi (u_2 \otimes v_1 ) & \varphi (u_2 \otimes v_2 ) & \cdots & \varphi (u_2 \otimes v_e )
\end{pmatrix}
\]
of linear forms on $\P^{2+e}$ defines $Y$ as its rank one locus.

Note that $S \equiv 2H + (4-d)F$ and $\deg (S \cap F) = (2H-F) \cdot F \cdot H = 2$. Thus $S \cap F$ is a plane conic on $F \cong \P^2$. Furthermore, the above multiplication map $\tau$ can be recovered from such a conic curve as follows:
\[
\varphi |_S : H^0 (S, \mathcal O_Y (F) |_S ) \otimes H^0 (S, \mathcal O_Y (H-F)|_S ) \to H^0 (\P^{2+e}, \mathcal O_{\P^{2+e}}(1)).
\]
Therefore the problem of finding a threefold rational normal scroll containing $S$ is equivalent to finding a plane conic contained in $S$. Based on these observations, we obtain the following result.

\begin{proposition}\label{prop:embedding scroll del Pezzo surfaces}
Let $S \subset \P^{2+e}$ be a smooth del Pezzo surface of degree $5 \le d \le 8$. Then

\renewcommand{\descriptionlabel}[1]%
             {\hspace{\labelsep}\textrm{#1}}
\begin{description}
\setlength{\labelwidth}{13mm} \setlength{\labelsep}{1.5mm}
\setlength{\itemindent}{0mm}
\item[{\rm (1)}] If $d=5$, then $\mu (S) = 5$. Moreover, if $Y \subset \P^5$ is a threefold of minimal degree which contains $S$ then $Y \cong S(1,1,1)$.

\item[{\rm (2)}] If $6 \le d \le 8$ and $S$ is the inner projection of $\nu_3 (\P^2) \subset \P^9$ from general $9-d$ points, then $\mu (S) = 9-d$. Moreover, if $Y \subset \P^d$ is a threefold of minimal degree which contains $S$ then $Y \cong S(1,2,d-5)$.

\item[{\rm (3)}] If $S = \nu_2 (\P^1 \times \P^1)$, then $\mu (S) = 2$. Moreover, if $Y \subset \P^8$ is a threefold of minimal degree which contains $S$ then $Y \cong S(2,2,2)$.
\end{description}
\end{proposition}

\begin{proof}
$(1)$ We may assume that $S$ is a blow-up of $\P^2 _{x_0,x_1,x_2}$ at the projective frame
\[
\Gamma = \{ P_0 = [1:0:0], P_1 = [0:1:0], P_2 = [0:0:1], P_3 = [1:1:1] \} \subset \P^2.
\]
We have the following set as a basis for $I(\Gamma)_3$:
\[
\{ x_0^2 x_1 - x_1 x_2^2, x_0 x_1^2 - x_1 x_2^2, x_0^2 x_2 - x_1 x_2^2, x_0 x_1 x_2 - x_1 x_2^2, x_1^2 x_2 - x_1 x_2^2, x_0 x_2^2 - x_1 x_2^2 \}
\]
These cubics passing through $\Gamma$ defines a rational map $\P^2 \dashrightarrow \P^5$ where the closure of the image is the blow-up $S \to \P^2$ along $\Gamma$. We need to describe all smooth rational normal threefold scrolls $Y$ which contain $S$.

Let $h$ be the pull-back of the hyperplane divisor of $\P^2$, and let $e_i$ be the exceptional divisor corresponding to $P_i \in \Gamma$ for each $0 \le i \le 3$. Then for any $Y$, the restriction $\mathcal O_Y (F)|_S$ must be the class of a plane conic, and hence either $\mathcal O_Y (F)|_S = \mathcal O_S (h-e_i)$ or $\mathcal O_S (2h - e_0 - e_1 -e_2 - e_3)$.

We first consider the case $\mathcal O_Y (F)|_S = \mathcal O_S (h-e_0)$. We have
$$H^0 (S, \mathcal O_S (h-e_0)) = \mathbb {K} \langle x_1, x_2 \rangle$$
and
$$H^0 (S, \mathcal O_S (2h - e_1 - e_2 - e_3)) = \mathbb {K}  \langle x_0^2 - x_1 x_2 , x_0 x_1 - x_1 x_2 , x_0 x_2 - x_1 x_2 \rangle,$$
and hence the multiplication map
\[
\tau : H^0 (S, \mathcal O_S (h-e_0)) \otimes H^0 (S, \mathcal O_S (2h - e_1 - e_2 - e_3)) \to H^0 (\P^5, \mathcal O_{\P^5}(1))
\]
is represented by the linear matrix
\begin{equation*}
\begin{CD}
M (\tau) & \quad = \quad & \begin{pmatrix}
\varphi (x_0^2 x_1 - x_1^2 x_2 ) & \varphi (x_0 x_1^2 - x_1^2 x_2 )   & \varphi ( x_0 x_1 x_2 - x_1^2 x_2 )\\
\varphi (x_0^2 x_2 - x_1 x_2^2 ) & \varphi (x_0 x_1 x_2 - x_1 x_2^2 ) & \varphi (x_0 x_2^2 - x_1 x_2^2 )
\end{pmatrix} \\
& = &  \begin{pmatrix}
 z_0 - z_4 & z_1 - z_4 & z_3 - z_4 \\
 z_2 & z_3 & z_5
 \end{pmatrix}.
\end{CD}
\end{equation*}
Its $2$-minors define a rational normal $3$-fold scroll $Y_0 \cong S(1,1,1)$ in $\P^5$. Similarly, we have $Y_1, Y_2, Y_3$ by choosing $\mathcal O_Y (F)|_S$ as $\mathcal O_S (h - e_i)$ for $i=1,2,3$.

Finally, suppose that $\mathcal O_Y (F)|_S = \mathcal O_S (2h - e_0 - e_1 - e_2 - e_3)$. Then we have
$$H^0 (S, \mathcal O_S(2h-e_0-e_1-e_2-e_3)) = \mathbb{K} \langle x_0 x_1 - x_1 x_2 , x_0 x_2 - x_1 x_2 \rangle$$
and
$$H^0 (S, \mathcal O_S (h)) = \langle x_0 , x_1, x_2 \rangle,$$
and hence the multiplication map
\[
\tau : H^0 (S, \mathcal O_S(2h-e_0-e_1-e_2-e_3)) \otimes H^0 (S, \mathcal O_S (h)) \to H^0 (\P^5, \mathcal O_{\P^5}(1))
\]
is represented by the matrix
\[
M(\tau) = \begin{pmatrix}
z_0 - z_3 & z_1 - z_4 & z_3 \\
z_2 - z_3 & z_3 - z_4 & z_5
\end{pmatrix}.
\]

\noindent $(2)$ It follows from the exactly same procedure. Note that the restriction $\mathcal O_Y (F) |_S$ must be $\mathcal O_S (h - e_i)$ for $0 \le i \le 8-d$ for these cases. We just list up the multiplication map for each $d=6,7,8$ and $\mathcal O_Y (F)|_S = \mathcal O_S (h - e_0)$ for simplicity.

\begin{itemize}
\item ($d=6$) The multiplication map is represented by the matrix
\[
M(\tau) =
 \begin{pmatrix}
 z_0 & z_1 & z_3 & z_4 \\
 z_2 & z_3 & z_5 & z_6
 \end{pmatrix},
\]
and thus we have three rational normal 3-fold scrolls $Y \cong S(1,1,2)$ containing $S$.

\item ($d=7$) The multiplication map is represented by the matrix
\[
M(\tau) = \begin{pmatrix}
 z_0 & z_1 & z_3 & z_4 & z_6 \\
 z_2 & z_3 & z_5 & z_6 & z_7
 \end{pmatrix},
\]
and thus we have two rational normal 3-fold scrolls $Y \cong S(1,2,2)$ containing $S$.

\item ($d=8$) The multiplication map is represented by the matrix
\[
M(\tau) = \begin{pmatrix}
z_0 & z_1 & z_2 & z_4 & z_5 & z_7 \\
z_3 & z_4 & z_5 & z_6 & z_7 & z_8
\end{pmatrix},
\]
and thus we have one 3-fold scroll $Y \cong S(1,2,3)$ containing $S$.
\end{itemize}

\noindent $(3)$ In the case $S = \nu_2 (\P^1 \times \P^1)$, we have $\operatorname{Pic}(S) = \Z h_1 \oplus \Z h_2$ where $h_1, h_2$ are the pull-backs of $\mathcal O_{\P^1}(1)$ by the first and the second projection $\P^1 \times \P^1 \to \P^1$. Since $S$ is embedded by $\mathcal O_S (2h_1 + 2h_2)$, we may write $S$ as the parametrized surface
\[
S = \{[x_0^2 y_0^2 : x_0^2 y_0 y_1 : x_0^2 y_1^2 : x_0 x_1 y_0^2 : x_0 x_1 y_0 y_1 : x_0 x_1 y_1^2 : x_1^2 y_0^2 : x_1^2 y_0 y_1 : x_1^2 y_1^2]\} \subset \mathbb{P}^8
\]
so that $[x_0 : x_1]$ and $[y_0 : y_1]$ run over $\P^1$. In this surface, there are exactly two classes of plane conics, namely, $\mathcal O_S (h_1)$ and $\mathcal O_S (h_2)$. As similar as above, in the case when $\mathcal O_Y (F)|_S = \mathcal O_S (h_1)$, we have the multiplication map
\[
\tau : H^0 (S, \mathcal O_S (h_1)) \otimes H^0 (\mathcal O_S (h_1 + 2h_2)) \to H^0 (\P^8 , \mathcal O_{\P^8}(1))
\]
which can be represented by the matrix
\[
M (\tau) = \begin{pmatrix}
z_0 & z_1 & z_2 & z_3&z_4 &z_5 \\
z_3&z_4&z_5 & z_6&z_7&z_8
\end{pmatrix},
\]
and thus we have two rational normal 3-fold scrolls $Y \cong S(2,2,2)$ containing $S$.
\end{proof}

\section{Proof of Theorem \ref{thm:main1}}\label{sect:MainThmProof}
\noindent This section is devoted to prove the Theorem \ref{thm:main1}. We first need to arrange the known cases about the varieties of almost minimal degree $X$ and $\mu (X)$ of them.

\begin{lemma}\label{lem:dim and codimension inequalities}
Let $X \subset \P^{n+e}$ be an $n$-dimensional smooth del Pezzo variety with $n \geq 2$ and $e \geq 3$. If $\mu (X) \geq 1$, then

\renewcommand{\descriptionlabel}[1]%
             {\hspace{\labelsep}\textrm{#1}}
\begin{description}
\setlength{\labelwidth}{13mm} \setlength{\labelsep}{1.5mm}
\setlength{\itemindent}{0mm}
\item[{\rm (1)}] $e \geq n+1$.
\smallskip

\item[{\rm (2)}] The Picard number of $X$ is at least two.
\smallskip

\item[{\rm (3)}] If $n \geq 3$ then $n = 3$ and $e$ is an even number.
\end{description}
\end{lemma}

\begin{proof}
Since $\mu (X) \geq 1$, there exists an $(n+1)$-fold variety $Y \subset \P^{n+e}$ of minimal degree which contains $X$. Then $Y$ is a smooth rational normal scroll by Lemma \ref{lem:divisor class of del Pezzo manifolds}.

\noindent $(1)$ We write $Y = \P_{\P^1} (\mathcal{E})$ as a projective bundle where
\begin{equation*}
\mathcal{E} = \mathcal{O}_{\P^1} (a_0) \oplus \mathcal{O}_{\P^1} (a_1) \oplus \cdots \oplus
\mathcal{O}_{\P^1} (a_n)
\end{equation*}
for some positive integers $a_0 \leq a_1 \leq \cdots \leq a_n$ such that $a_0 + a_1 + \cdots + a_n = e$. This shows that $e \geq n+1$.

\noindent $(2)$ Let $\pi : Y \rightarrow \P^1$ be the projection map. Then the divisors in $O_X (X \cap F )$ of $X$ are pairwise disjoint. This is impossible if the Picard number of $X$ is equal to one. Thus the Picard number of $X$ must be at least two.

\noindent $(3)$ By Lemma \ref{lem:divisor class of del Pezzo manifolds}, we have $X \equiv 2H+(2-e)F$. In particular, the line bundle $\mathcal{O}_Y (2H+(2-e)F)$ must have an irreducible section defining $X$. On the other hand, if $2a_{n-1} +2-e$ is negative then every section in $| \mathcal{O}_Y (2H+(2-e)F) |$ is reducible since
\begin{equation*}
h^0 (\P^1 ,\mathcal{O}_{\P^1} (a_i + a_j +2-e)) = 0 \quad \mbox{whenever $0 \le i \le j \le n-1$.}
\end{equation*}
Therefore $2a_{n-1} +2-e \geq 0$ and hence it follows that
\begin{equation*}
2 \geq a_0 + \cdots + a_{n-2} - a_{n-1} + a_n \geq n-1
\end{equation*}
Thus $n$ is at most $3$. In particular, if $n \geq 3$ then $n=3$.

Suppose that $n =3$. Then it holds that $\frac{e}{2} -1 \leq a_2 \leq a_3$ and hence
\begin{equation*}
e = a_0 +a_1 + a_2 + a_3 \geq 2 + (\frac{e}{2} -1)+(\frac{e}{2} -1) =e.
\end{equation*}
Thus we have $a_0 = a_1 = 1$ and $a_2 = a_3 = \frac{e}{2} -1$. In particular, $e$ must be even.
\end{proof}

\begin{corollary}\label{cor:delta zero}
If $X \subset \P^{n+e}$ is either $\nu_3 (\P^2 ) \subset \P^9$ or $\nu_2 (\P^3 ) \subset \P^9$ or ${\rm Bl}_p (\P^3 ) \subset \P^8$ or $\P^2 \times \P^2 \subset \P^8$ or else a $k$-dimensional smooth linear section of $\mathbb{G} (1,\mathbb{P}^4) \subset \P^9$ for some $3 \leq k \leq 6$, then $\mu (X) = 0$.
\end{corollary}

\begin{proof}
If $X$ is $\nu_3 (\P^2 )$ or $\nu_2 (\P^3 )$ or a $k$-dimensional smooth linear section of $\mathbb{G} (1,\mathbb{P}^4) \subset \P^9$ for some $3 \leq k \leq 6$, then its Picard number is equal to one. Therefore $\mu (X) = 0$ by Lemma \ref{lem:dim and codimension inequalities}-$(2)$. If $X$ is ${\rm Bl}_p (\P^3 )$ in $\P^8$, then $e = 5$ is odd and hence $\mu (X) = 0$ by Lemma \ref{lem:dim and codimension inequalities}-$(3)$. If $X$ is $\P^2 \times \P^2$ in $\P^8$, then $n = 4$ and hence $\mu (X) = 0$ by Lemma \ref{lem:dim and codimension inequalities}-$(3)$.
\end{proof}

Now, we are ready to prove Theorem \ref{thm:main1}.\\

\noindent {\bf Proof of Theorem \ref{thm:main1}.} Throughout the proof, we assume that $X$ is not a cone since $\mu (X) = \mu(X_0)$ if $X$ is a cone over a variety $X_0$.
\smallskip

\noindent $(i) \Rightarrow (ii)$ : If $d=e+1$, then $\mu (X) > 0$. Indeed, choose a smooth point $p$ of $X$. Then $Y = \mbox{Join} (p,X)$ is an $(n+1)$-dimensional variety of minimal degree which contains $X$. Thus $\mu (X) >0$.

If $d\geq e+3$ and $\beta_{e-1,1} (X) \neq 0$ then $\mu (X) >0$ by Theorem \ref{thm:Kp1}. Also, if $d = e+2$ and either $n=1$ or $X$ is singular then $\mu (X) > 0$ by Proposition \ref{prop:first computation of mu(X)}.

Now, suppose that $\beta_{e-1,1} (X) \neq 0$ and $\mu (X) = 0$. Then from the previous explanation, $X$ must be a smooth del Pezzo variety of dimension $n \geq 2$. We use the classification of smooth del Pezzo varieties in Theorem \ref{thm:del Pezzo}. If $n =2$ and $\mu (X) = 0$, then $X$ must be $\nu_3 (\P^2 )$ by Proposition \ref{prop:embedding scroll del Pezzo surfaces}. If $n = 3$ and $\mu (X) = 0$, then $X \neq \P^1 \times \P^1 \times \P^1$ by Proposition \ref{prop:del Pezzo 3fold Segre}. This completes the proof.
\smallskip

\noindent $(ii) \Rightarrow (i)$ : Since all varieties $X$ listed in $(ii)$ are del Pezzo, it holds that $\beta_{e-1,1}(X) >0$. Also, by Corollary \ref{cor:delta zero}, we have $\mu (X) = 0$ except the case where $X$ is the threefold $\P_{\P^2} (\mathcal T_{\P^2})$ in $\P^7$. Thus it remains to show that $\mu ( \P_{\P^2} (\mathcal T_{\P^2}) ) = 0$. To this aim, let us suppose that $X=\P_{\P^2} (\mathcal T_{\P^2})$ is contained in a fourfold $Y \subset \P^7$ of minimal degree. Then $Y$ must be smooth by Proposition \ref{prop:smoothness of embedding scroll}. Therefore $Y = S(1,1,1,1)$ and thus
\begin{equation*}
Y = S(1,1,1,1) = \P^3 \times \P^1 \subset \P^7  \quad \mbox{and} \quad X \equiv 2H-2F
\end{equation*}
(cf. Lemma \ref{lem:divisor class of del Pezzo manifolds}). Furthermore,
\begin{equation*}
H^0 (Y,\mathcal{O}_{Y} (2H-2F)) \cong \bigoplus_{0 \leq i \leq j \leq 3} H^0 (\P^1 ,\mathcal{O}_{\P^1} ) {x_i x_j}
\end{equation*}
and hence every divisor in $|\mathcal{O}_{Y} (2H-2F)|$ is the product of $\P^1$ and a quadric in $\P^3$. Therefore, $X = \P^1 \times Q$ where $Q \subset \P^3$ is a smooth quadric, in other words, $X=\P^1 \times \P^1 \times \P^1$. This is a contradiction and hence $\mu (X) = 0$.  

\noindent $(ii) \Leftrightarrow (iii)$ By Theorem \ref{thm:del Pezzo}, del Pezzo varieties appeared in the list $(ii)$ are exactly characterized by the condition that their Picard numbers are at most $n-1$.
\qed\\

\section{Embedding scrolls as syzygy schemes}\label{sect:SyzygyScheme}
\noindent 
In this section, we provide an alternative idea to understand Theorem \ref{thm:main1} via syzygy schemes. M. Green first introduced a geometric object associated to a Koszul cohomology class $\gamma \in \mathcal K_{p,1} (R(X), R_1)$ which is defined by minimal quadrics to represent $\gamma$ \cite{Gr1}. Such an object is often called the \emph{syzygy scheme}, and it is surprisingly useful to understand the structure of quadrics in $I(X)$ when either $X$ has nonvanishing $\mathcal K_{p,1} (R(X), R_1)$ with high $p$, or some Koszul classes in $\mathcal K_{p,1} (R(X),R_1)$ having small syzygy ranks compared to $p$, see for instance \cite{Eh, GvB, AE}. In this section we will frequently use $\mathcal K_{p-1,q+1} (I(X),R_1) \simeq \mathcal K_{p,q} (R(X), R_1)$ since $X$ is non-degenerate. 



\begin{definition}[\cite{Gr1, Eh, AN}]
Let $X \subset \mathbb{P}^{n+e}$ be as above, and let $R = \mathbb{K}[x_0, \cdots, x_{n+e}]$ be the homogeneous coordinate ring of $\mathbb{P}^{n+e}$. Let $\gamma \in  \mathcal K_{p-1,2} (I(X), R_1)$
be a nonzero Koszul cohomology class minimally represented as
\[
\gamma = \sum_{|J|=p-1} v_J \otimes q_J
\]
where $v_J \in \bigwedge^{p-1} R_1$ and $q_J \in I(X)_2$. The number of summands is called the \emph{rank} of the syzygy $\gamma$. We define the \emph{syzygy ideal of $\gamma$} as the ideal generated by the quadrics $q_J$, and \emph{the syzygy scheme $Syz(\gamma)$} as the scheme defined by the syzygy ideal of $\gamma$.
\end{definition}


We first observe that syzygy schemes provide a simple explanation on $(n+1)$-folds $Y$ of minimal degree containing $X$ for many cases as follows. The idea of the proof is essentially due to M. Green \cite[Section 3-(c)]{Gr2}. 

\begin{theorem}\label{thm:uniqueness and syzygy scheme}
Let $X \subset \P^{n+e}$ be a non-degenerate projective irreducible variety of dimension $n \ge 1$, codimension $e \geq 3$ and degree $d$. Suppose that $\beta_{e-1,1} (X) \neq 0$. If either $d=e+2$ and $X$ is not arithmetically Cohen-Macaulay, or else $d \geq e+3$, then $\mu (X) =1$ and thus there exists exactly one $(n+1)$-dimensional variety $Y$ of minimal degree which contains $X$.
\end{theorem}

\begin{proof}
By Theorem \ref{thm:Kp1} and Theorem \ref{thm:BS}, we already know that $\mu (X) >0$, indeed, $X$ is contained in an $(n+1)$-fold $Y $ of minimal degree. From this inclusion, we have a nonzero class $\gamma \in \mathcal K_{e-2,2} (I(Y), R_1) \subset \mathcal K_{e-2,2} (I(X), R_1)$. Note that the syzygy scheme $Syz(\gamma)$ is defined by $\ge \binom{e}{2}$ independent quadrics in $I(Y)$, cf. \cite{Gr1} and \cite[Proposition 3.26]{AN}. Since $\dim I(Y)_2 = \binom{e}{2}$, the only possibility is that $Syz(\gamma)$ and $Y$ coincide for any nonzero Koszul class $\gamma \in \mathcal K_{2-2,2} (I(Y),R_1)$.

When $d=e+2$ and $X$ is not arithmetically Cohen-Macaulay, Theorem \ref{thm:BS} implies that $\beta_{e-1,1}(X) = \beta_{e-1,1}(Y) = e-1$, in particular, $\mathcal K_{e-2,2} (I(Y), R_1) = \mathcal K_{e-2,2} (I(X), R_1)$. Indeed, if there is another $(n+1)$-fold $Z$ of minimal degree which contains $X$, then any nonzero element $\gamma^{\prime} \in \mathcal K_{e-2,2} (I(Z),R_1)$ will define the syzygy scheme
\[
Z = Syz(\gamma^{\prime}) = Syz(\gamma) = Y.
\]
We conclude that $\mu (X) = 1$.

The same argument also holds when $d \ge e+3$ since we have $\beta_{e-1,1}(X) = \beta_{e-1,1} (Y) = e-1$, and hence $\mathcal K_{e-2,2} (I(Y), R_1) = \mathcal K_{e-2,2} (I(X), R_1)$ as in \cite[Theorem 3.5]{NP}.
\end{proof}



We now focus on smooth and arithmetically Cohen-Macaulay varieties of almost minimal degree (= smooth del Pezzo varieties). We first describe the syzygy scheme of a Koszul class $\gamma \in \mathcal K_{e-2,2}(I(X),R_1)$ when $X$ is a smooth del Pezzo variety of maximal dimension (= not a smooth linear section of a higher dimensional del Pezzo variety of the same degree). In these cases, we see that $Syz(\gamma)$ is not an $(n+1)$-fold of minimal degree containing $X$. The following observation is based on the computations using Macaulay2.

\begin{proposition}\label{prop:SyzGen}
Let $X$ be a smooth del Pezzo varieties of degree $5 \le d \le 9$ (so that the codimension $e=d-2$) of maximal dimension. Then, the syzygy scheme $Syz(\gamma)$ of a general Koszul class $\gamma \in K_{e-2,2} (I(X), R_1)$ is
\begin{enumerate}
\item $(d=5)$ $X = \mathbb{G}(1, \mathbb{P}^4) \subset \mathbb{P}^9$ and $Syz(\gamma)$ is the union of $X$ and a linear subspace $\Lambda \cong \mathbb{P}^5$ intersect along a quadric $4$-fold defined by the Pl\"{u}cker relation on the $6$ coordinates of $\Lambda$.

\item $(d=6)$ $X$ is either a Segre fourfold $\mathbb{P}^2 \times \mathbb{P}^2 \subset \mathbb{P}^8$ and $Syz (\gamma)$ is the union of $X$ and a point  outside $X$, or $X = \mathbb{P}^1 \times \mathbb{P}^1 \times \mathbb{P}^1 \subset \mathbb{P}^7$ and $Syz (\gamma)$ is a union of $X$ and a line intersecting at two points.

\item $(d=7)$ $X$ is an inner projection of $v_2 (\mathbb{P}^3)$ into $\mathbb{P}^8$ and $Syz(\gamma) = X$.

\item $(d=8)$ $X = v_2 (\mathbb{P}^3) \subset \mathbb{P}^9$ and $Syz(\gamma) = X$.
\item $(d=9)$ $X = v_3 (\mathbb{P}^2) \subset \mathbb{P}^9$ and $Syz(\gamma) = X$.
\end{enumerate}

In any case, we have $\displaystyle Syz_{e-1}(X) = \bigcap_{\gamma \in K_{e-2,2} (I(X),R_1)} Syz (\gamma) = X$.
\end{proposition}

\begin{proof}
The statement comes from explicit computations of the syzygy scheme $Syz(\gamma)$. In any case, the syzygy scheme $Syz(\gamma)$ becomes a scheme defined by $\beta_{1,1}(X)-1$ quadrics in $I(X)$. In particular, $Syz(\gamma)$ is a (ideal-theoretic) union of $X$ and an ideal defined by numbers of linear forms in $\mathbb{P}^{n+e}$. We explain briefly why this computations make sense.

Note that $X$ is arithmetically Gorenstein, in particular, $R(X)$ has a ``self-dual'' minimal $R$-free resolution of the form
\[
F_{\bullet} : 0 \to R(-e-2) \stackrel{d_e} \to R(-e)^{\beta_{e-1,1}} \stackrel{d_{e-1}} \to \cdots \stackrel{d_3} \to R(-3)^{\beta_{2,1}} \stackrel{d_2} \to R(-2)^{\beta_{1,1}} \stackrel{d_1} \to R \to 0
\]
so that $\operatorname{coker} d_1 = R(X)$. Here, ``self-dual'' means that the above resolution is symmetric in the sense that the dual complex $F_{\bullet}^{\vee} = \Hom_{R} (F_{\bullet}, R(-e-2))[-e]$ 
is isomorphic to $F_{\bullet}$. 
For $2 \le p \le e-1$, consider the exterior product
\[
D_p = d_2 \wedge d_3 \wedge \cdots \wedge d_p
\]
so that the $\left( \beta_{p,1} \times \beta_{1,1}  \right)$-matrix $D_p$ encodes all the $p$-th syzygies $\mathcal K_{p-1,2}(I(X),R_1)$. In particular, a $p$-th syzygy $\gamma = \sum_{i=1}^{\beta_{1,1}} x_{J_i} \otimes q_i \in \mathcal K_{p-1,2} (I(X), R_1)$ means that $(x_{J_i} \in \wedge^{p-1} R_1)_{i=1}^n$ consisting $\gamma$ lies in the row space of $D_p$ and vice versa.

Since $X$ is a del Pezzo variety, we have a square matrix $D_{e-1}$ which is an $N \times N$ matrix where $N = \beta_{1,1} = \beta_{e-1,1} =  \binom{e+1}{2} -1$ so that the columns of $D_{e-1}$ defines $(e-1)$-th syzygies in $\mathcal K_{e-2,2}(I(X),R_1)$. By taking suitable bases, we observe that $D_{e-1}$ is skew-symmetric, not only in the case of $e=3$ \cite{BE} but also for higher codimension $e \ge 4$. We refer interested readers to \cite[Section 2.5]{R} and \cite[Theorem 1.5, 4.2]{St} for more details on the structure theorem for arithmetically Gorenstein varieties. 


In particular, the $i$-th row ($1 \le i \le \beta_{e-1,1}$) of $D_{e-1}$ indicate an $(e-1)$-th syzygy among the quadrics $q_1, \cdots, \hat{q_i}, \cdots, q_{N} \in I(X)$. A general Koszul class $\gamma \in \mathcal  K_{e-2,2}(I(X),R_1)$ is a linear combination of such rows, and hence also defines the syzygy scheme $Syz(\gamma)$ defined by $N-1 = \binom{e+1}{2} - 2$ quadrics. In most cases, the ideal $I(\gamma)$ generated by these $N-1$ quadrics is not saturated, however, $I(\gamma)$ corresponds to a union of $X$ and the linear subspace defined by linear forms involved in $\gamma$: see \cite{GvB} for an explicit description for the case $d=5$. 
\end{proof}

Now we discuss syzygy schemes arose from arbitrary smooth del Pezzo varieties. Let $\Lambda \subset \P^{n+e}$ be a hyperplane cut out by a linear form $t \in R_1$, let $H=\Lambda \cap X$ be the hyperplane section, and let $W=R_1/ \langle t \rangle$. For a nonzero Koszul class $\gamma \in \mathcal K_{e-1,2} (I(X), R_1)$, we consider its restriction $\gamma |_H \in \mathcal K_{e-2,2} (I(H), W)$. It is clear that $Syz(\gamma) \cap \Lambda \subseteq Syz(\gamma |_H)$, but the equality does not hold in general even when $\mathcal K_{e-2,2} (I(X),R_1) \cong \mathcal K_{e-2,2}(I(H),W)$ which makes hard to describe each $Syz(\gamma |_H )$ explicitly. However, we can still show the following statement.

\begin{proposition}\label{prop:SyzDelPezzo}
Let $X \subset \mathbb{P}^{n+e}$ be a smooth del Pezzo variety of codimension $e \ge 3$. Let $0 \neq \gamma \in \mathcal K_{e-2,2} (I(X), R_1)$ be a nonzero Koszul cohomology class.

\begin{enumerate}
\item If $\gamma$ has rank $e+1$, then $Syz(\gamma)$ is an $(n+1)$-dimensional variety of minimal degree containing $X$. In other words, $\gamma$ is contained in the subset $\mathcal K_{e-2,2}(I(Y), R_1) \subseteq \mathcal K_{e-2,2}(I(X),R_1)$ for some $(n+1)$-dimensional variety of minimal degree $Y \supset X$.

\item $\displaystyle Syz_{e-1} (X)  = \bigcap_{\gamma \in \mathcal K_{e-2,2} (I(X),R_1)} Syz (\gamma) = X$.
\end{enumerate}
\end{proposition}

\begin{proof}
The first statement is just \cite[Corollary 5.2]{GvB}. 

For the second statement, consider a general linear subspace $\Lambda \cong \P^e$ so that $Z = X \cap \Lambda$ is a set of $d=e+2$ points in linearly general position. Since $X$ is arithmetically Cohen-Macaulay, we have $\mathcal K_{e-2,2}(I(X),R_1) = \mathcal K_{e-1,2} (I(Z),W)$ where $W$ is the quotient of $R_1$ by the generators of $\Lambda$ and $Syz(\gamma) \cap \Lambda \subseteq Syz(\gamma |_{Z})$. Hence, $X \cap \Lambda \subseteq Syz_{e-1} (X) \cap \Lambda \subseteq Syz_{e-1}(Z)$.

If we add another point  $Q \in \Lambda$ so that $Z \cup \{Q\}$ is still in linearly general position, then there is a unique rational normal curve $C_Q$ passing through points in $Z$ and $Q$ thanks to the strong Castelnuovo lemma. In particular, there is $\gamma_Q \in \mathcal K_{e-2,2}(I(C_Q), W) \subseteq \mathcal K_{e-2,2} (I(Z), W)$ so that $Syz(\gamma_Q) = C_Q$. By choosing different points $Q$ (and thus different rational normal curves $C_Q$), we conclude that $Z \subseteq Syz_{e-1} (Z) \subseteq \bigcap_{Q} Syz(\gamma_Q) = Z$. In particular, $Syz_{e-1} (X) \cap \Lambda = Syz_{e-1} (Z) = Z$ holds for any choice of general linear subspace $\Lambda \cong \P^e$. We conclude that $Syz_{e-1}(X)=X$ as desired.
\end{proof}

\begin{remark}
\begin{enumerate}
\item If $X$ is contained in an $(n+1)$-fold $Y$ of minimal degree, then $Syz(\gamma)$ can recover $Y$ for suitable choice of $\gamma$. In particular, linear sections of del Pezzo varieties can have $\mu(X)>0$. 

Consider the case when $d=5$ so that $X$ is generated by $5$ quadrics obtained as $4$-Pfaffians of a $5 \times 5$ skew-symmetric matrix. In the case, a nonzero Koszul class $\alpha \in \mathcal K_{e-2,2} (I(X), R_1)$ has rank $ \le 4$, that is, we may find $4$ linear forms in $R_1$ and $4$ quadrics in $I(X)$ so that
\[
\alpha = \ell_1 \otimes q_1 + \cdots + \ell_4 \otimes q_4, \quad \sum_{i=1}^4 \ell_i \cdot q_i = 0.
\]
If there is a nonzero Koszul class $\gamma \in \mathcal K_{1,2}(I(X),R_1)$ of rank $3$ (such a Koszul class is called \emph{of scrollar type}), then $Y=Syz(\gamma)$ becomes an $(n+1)$-fold of minimal degree containing $X$, cf. \cite[Corollary 5.2]{GvB}. In the case, $Y$ is defined as $2$-minors of a $2 \times 3$ matrix of linear form, and thus $I(Y)$ is generated by three quadrics of rank $\le 4$. On the other hand, there are not enough quadrics of rank $\le 4$ in $I(X)$ when $\dim X \ge 3$, so there is no chance to have a Koszul class of scrollar type in $\mathcal K_{1,2}(I(X),R_1)$. 

For instance, if we take a hyperplane section $H$ of $X$ cut out by a linear form $\ell_1$ appeared in the minimal expression for $\alpha$ in the above, we have
\[
\alpha |_H = (\ell_2 \otimes q_2 + \ell_3 \otimes q_3 + \ell_4 \otimes q_4)|_H
\]
so that $\alpha|_H$ is a syzygy of rank $3$. In particular, the Koszul class $\alpha|_H \in \mathcal K_{1,2} (I(H), R_1 / \langle \ell_1 \rangle)$ determines an $n$-fold of minimal degree containing $H = X \cap V(\ell_1)$. However, such a hyperplane section $H$ can be smooth only when $\dim H \le 2$. 
We leave in Section \ref{sect:Grassmannian} for the computation of the number of independent quadrics of given rank in $I(X)$ where $X$ is a general linear section of $\mathbb{G} (1, \mathbb{P}^4)$.

\item Syzygy schemes could explain the cases when $X$ is a smooth del Pezzo variety such that $\mu(X)>0$. Suppose that $\mathcal O_X(1)$ admits a decomposition $\mathcal O_X(1) = A \otimes B$ with $h^0 (A) = r_1 + 1 \ge 2$ and $h^0 (B) = r_2 + 1 \ge 2$. The Green-Lazarsfeld nonvanishing theorem tells us that there is a nonzero Koszul class $0 \neq \gamma \in \mathcal K_{r_1 + r_2 - 1, 1}(X, \mathcal O_X(1))$. It is well-known that a Green-Lazarsfeld class $\gamma$ has rank $\le r_1 + r_2 + 1$, and $\gamma$ is of scrollar type if and only if one can decompose $\mathcal O_X(1) = A \otimes B$ with $h^0(A) = 2$. If $X$ is a blow-up of $\mathbb{P}^2$ at $\le 4$ points embedded anticanonically, then one may take $A$ as $\mathcal O_X (h - e_0)$ where $h$ is the strict transform of a line in $\mathbb{P}^2$ and $e_0$ is an exceptional divisor corresponding to a single point in the blow-up center. If $X \simeq \mathbb{P}^1 \times \mathbb{P}^1$ embedded by $\mathcal O_{\mathbb{P}^1 \times \mathbb{P}^1} (2,2)$, then one may take $A$ as a ruling $\mathcal O_{\mathbb{P}^1 \times \mathbb{P}^1}(1,0)$. On the other hand, the triple Veronese surface $v_3 (\mathbb{P}^2) \subset \mathbb{P}^9$ is embedded by $\mathcal O_X(1) = \mathcal O_{\mathbb{P}^2} (3)$ does not admit a decomposition $\mathcal O_X(1) = A \otimes B$ with $h^0(A) = 2$, so we cannot expect a Green-Lazarsfeld class of scrollar type in this case. We refer interested readers to \cite{Eis88, SS} for more details on these approaches, related problems, and applications. 
\end{enumerate}
\end{remark}

\section{Further discussion on quintic del Pezzo manifolds of dimension $\geq 3$}\label{sect:Grassmannian}
\noindent Let $X_k \subset \P^{k+3}$ be a smooth del Pezzo variety of degree $5$ of dimension $k \geq 3$. Thus $X_k$ is a smooth linear section of the Grassmannian $\mathbb{G} (1,\P^4)$ in $\P^9$. By Theorem \ref{thm:main1}, we have

\begin{theorem}\label{thm:quintic del Pezzo manifolds}
Let $X_k \subset \P^{k+3}$ be as above. Then $\beta_{2,1} (X_k ) > 0$ and $\mu (X_k ) = 0$.
\end{theorem}

Lemma \ref{lem:dim and codimension inequalities} gives three reasons why the statement of Theorem \ref{thm:quintic del Pezzo manifolds} is true. Namely, each of the following three properties of $X_k$ is an obstruction of the existence of a $(k+1)$-fold of minimal degree which contains $X_k$.\\

\begin{enumerate}
\item[$(i)$] ${\rm codim} (X_k) = 3 < k+1$;
\item[$(ii)$] The Picard number of $X_k$ is equal to one;
\item[$(iii)$] $k \geq 3$ and ${\rm codim} (X_k) = 3$ is odd.  \\
\end{enumerate}

\noindent Along this line, the main purpose of this section is to give alternative proof of the fact that $\mu (X_k ) = 0$ for all $3 \leq k \leq 6$ using the following rank property of the quadratic polynomials in the homogeneous ideal of $X_k$.

\begin{definition}
Let $X\subset \mathbb{P}^r$ be a non-degenerated projective irreducible variety. For each $t \geq 3$, we define $\delta(X,t)$ as the maximum number of linearly independent quadrics of rank $\leq t$ in $I(X)_2$.
\end{definition}

\begin{proposition}\label{prop:quintic del Pezzo, rank}
Let $X_k \subset \P^{k+3}$ be a smooth del Pezzo variety of degree $5$ of dimension $k \geq 1$. Then

\renewcommand{\descriptionlabel}[1]%
             {\hspace{\labelsep}\textrm{#1}}
\begin{description}
\setlength{\labelwidth}{13mm} \setlength{\labelsep}{1.5mm}
\setlength{\itemindent}{0mm}
\item[{$(1)$}] $\delta (X_1 , 3) = 5$.

\item[{$(2)$}] $\delta (X_2 , 3) = 0$ and $\delta (X_2 , 4 ) = 5$.

\item[{$(3)$}] $\delta (X_3 , 4) = 0$ and $\delta (X_3 , 5 ) = 5$.

\item[{$(4)$}] $\delta (X_4 , 4) = 0$, $\delta (X_4 , 5 ) = 3$ and $\delta (X_4 , 6 ) = 5$.

\item[{$(5)$}] $\delta (X_5 , 4) = 0$, $\delta (X_5 , 5 ) = 1$ and $\delta (X_5 , 6 ) = 5$.

\item[{$(6)$}] $\delta (X_6 , 5 ) = 0$ and $\delta (X_6 , 6 ) = 5$.
\end{description}
\end{proposition}

\begin{proof}
Note that $I(X_k )$ is minimally generated by five quadratic polynomials, say $Q_1 , \ldots , Q_5$. Thus a general member $Q$ of $I(X_k )_2$ is of the form
\begin{equation*}
Q = x_1 Q_1 + \cdots + x_5 Q_5 \quad \mbox{for some} \quad x_1 , \ldots , x_5 \in \mathbb{K}.
\end{equation*}
We identify $\P (I(X)_2 )$ with $\P^4$ by sending $[Q]$ to $[x_1 , \ldots , x_5 ]$. Now, let $M_k$ be the $(k+4) \times (k+4)$ symmetric matrix associated with $Q$. Its entries are linear forms in $x_1 , \ldots , x_5$. For each $1 \leq t \leq r$, we define the locus
\begin{equation*}
\Phi_t (X_k ) := \{ [Q] ~|~ Q \in I(X_k )_2 - \{0\}, ~\mbox {rank} (Q) \leq t \} \subset \P ( I(X_k )_2 ) = \P^4
\end{equation*}
of all quadratic polynomials of rank $\leq t$ in the projective space $\P^4$. Thus $\Phi_t (X_k )$ is a projective algebraic set in $\P^4$ and its homogeneous ideal is the radical of $I(t+1,M_k )$. In particular, $\delta(X_k ,t)$ is one more than the dimension of the linear subspace of $\P^4$ spanned by $\Phi_t (X_k )$. Therefore it holds that
$$\delta(X_k ,t) = 5 - (\mbox{the maximal $\sharp$ of linearly independent linear forms in} \sqrt{I(t+1,M_k )}).$$

For $(1)$, we refer the reader to \cite[Theorem 1.1]{P2022} which addresses that every elliptic normal curve $\mathcal{C} \subset \P^n$ of degree $\geq 4$ has the property that its homogeneous ideal is generated by quadratic polynomials of rank $3$. Since $X_1 \subset \P^{4}$ is an elliptic normal curve of degree $5$, it holds that $\delta (X_1 , 3) = 5$.\\

For $(2)$, $X_2\in \mathbb{P}^5$ be a smooth del Pezzo surface of degree $5$. There is up to isomorphism only one such surface, given by blowing up the projective plane in $4$ points with no $3$ on a line. Let $\mathbb{K}[z_0\ldots z_5]$ be the coordinate ring of $\mathbb{P}^5$. It is known that the generators of $I(X_2)$ are given by the $5$ Pfaffians of the $4\times 4$ diagonally symmetric submatrices of the matrix
$$\begin{pmatrix}
0 & -z_0+z_1 & -z_1 & z_1-z_5 & z_5\\
z_0-z_1 & 0 & -z_2 & -z_5 & z_5\\
z_1 & z_2 & 0 & z_2 & -z_3\\
-z_1+z_5 & z_5 & -z_2 & 0 & z_4\\
-z_5 & -z_5 & z_3 & -z_4 & 0
\end{pmatrix},$$
which is 
\begin{align*}
& Q_1=-z_{0}z_{2}+(-z_{1}+z_{2})z_{5} &Q_2=(z_{0}-z_{1})z_{3}+(z_{1}-z_{2})z_{5},\\
& Q_3=(-z_{0}+z_{1})z_{4}-z_{1}z_{5} & Q_4=(z_{3}-z_{4})z_{1}+(z_{2}-z_{3})z_{5},\\
& Q_5=(-z_{4}+z_{5})z_{2}-z_{3}z_{5}
\end{align*}
Let $Q=x_1Q_1+\cdots+x_5Q_5$ be a quadric element in $I(X_2)$. Using \emph{Macaulay2}, we can show that 
$$\sqrt{I(4,M_2)}=(x_1,x_2,x_3,x_4,x_5).$$
Hence $\delta(X_2,3)=5-5=0$. Since we can easily check that each generator $Q_i$ has rank $4$, we can show that $\delta(X_2,4)=5$.

For $(3) \sim (6)$, we compute $\delta (X_k , t)$ using \emph{Macaulay2}. First, let $\{p_{ij} ~|~ 0 \le i < j \le 4 \}$ be the set of homogeneous coordinates of $\P^9$. Then the homogeneous ideal of $\mathbb{G}(1,\mathbb{P}^4)$ is generated by the following five Pl\"{u}cker relations
\begin{equation*}
Q(i,j,k,l) = p_{ij} p_{kl} - p_{ik} p_{jl} + p_{jk} p_{il}  \quad (0 \le i < j < k < l \le 4).
\end{equation*}
In particular, $\delta (X_6 , 6) = 5$. On the other hand, $\delta (X_6 , 5) = 0$ by \cite[Theorem 5.1]{MP}.

Now, let $H_1$, $H_2$ and $H_3$ be hyperplanes of $\P^9$ which are defined respectively by
\begin{equation*}
p_{34}-p_{01}-p_{02} , ~ p_{24}-p_{03}-p_{04} \quad \mbox{and} \quad p_{23}-p_{12}-p_{13}-p_{14} .
\end{equation*}
Then one can check that $X_5 := \mathbb{G} (1,\mathbb{P}^4) \cap H_1$, $X_4 := X_5 \cap H_2$ and $X_3 := X_4 \cap H_3$ are all irreducible and smooth. Also the homogeneous ideal of $X_k$ is generated by the restrictions of the above five Pl\"{u}cker relations to the projective spaces $\langle X_k \rangle = \P^{k+3}$. In particular, it holds that $\delta (X_k , 6) = 5$ for all $k \geq 2$ since $\delta (X_6 , 6) = 5$.

We can compute $\delta (X_k , t)$ by using these five quadratic polynomials. For example, the homogeneous ideal of $X_5$ is generated by $Q^{(1)}(i,j,k,l)$'s which are the substitution of the Pl\"{u}cker relations by $p_{34}=p_{01}+p_{02}$. Let
\begin{equation*}
Q^{(1)} = \sum_{0 \le i < j < k < l \le 4} x^{(1)}_{i,j,k,l} Q^{(1)}(i,j,k,l)
\end{equation*}
be a quadric element in the homogeneous ideal of $X_5$. Using Macaulay2, we can show that
$$\sqrt{I(6,M_5 )} =(x_{0123}^{(1)},x_{0124}^{(1)},x_{0234}^{(1)}-x_{0134}^{(1)},x_{1234}^{(1)}).$$
Since it has $4$ linearly independent linear form, we conclude that $\delta (X_5 , 5) = 1$.

All the remaining results can be obtained by similar method.
\end{proof}

\noindent {\bf Alternative proof of Theorem \ref{thm:quintic del Pezzo manifolds}.}
Suppose that $k \geq 3$ and $X_k$ is contained in a $(k+1)$-dimensional variety $Y$ of minimal degree. Then $Y$ must be a rational normal scroll and hence $I(Y)$ is generated by quadratic polynomials of rank at most $4$. Since $I(X_k )$ contains $I(Y)$, it follows that
$$\delta (X_k , 4 ) \geq \delta (Y , 4 ) \geq 1.$$
On the other hand, $\delta (X_k , 4 ) = 0$ by Proposition \ref{prop:quintic del Pezzo, rank}. This completes the proof that there does not exist such a variety $Y$ and hence $\mu (X_k ) = 0$.  \qed


\ \\
\section*{ Statements and Declarations}
\ \\
{ \bf Acknowledgement : } 
Y. K.  thanks Frank-Olaf Schreyer for helpful discussion, particularly on the analysis for computational results in Proposition \ref{prop:SyzGen}. The first named author was supported by the Basic Science Program of the NRF of Korea (NRF-2022R1C1C1010052). The second named author was supported by a KIAS Individual Grant(MG083101) at Korea Institute for Advanced Study. The third named author was supported by the National Research Foundation of Korea(NRF) grant funded by the Korea government(MSIT) (No. 2022R1A2C1002784).
\ \\
{\bf Competing Interests : }The authors declare that they have no known competing financial interests or personal relationships that could have appeared to influence the work reported in this paper.



\begin{thebibliography}{0000000}
\bibitem[AE]{AE} M. Aprodu and E. Sernesi, {\em Secant spaces and syzygies of special line bundles on curves}, Alg. Number Theory 9 , no. 3, 585--600(2015).

\bibitem[AN]{AN} M. Aprodu and J. Nagel, {\em Koszul cohomology and algebraic geometry}, Univ. Lecture Ser., 52, American Mathematical Society, Providence, RI, (2010)

\bibitem[BP1]{BP1} M. Brodmann and E. Park, {\em On varieties of almost mnimal degree I: Secant loci of rational normal scrolls}, J. Pure and Appl. Algebra 214 , 2033--2043(2010).

\bibitem[BP2]{BP2} M. Brodmann and E. Park, {\em On varieties of almost minimal degree III: Tangent spaces and embedding scrolls}, J. Pure and Appl. Algebra 215 , 2859--2872(2011).

\bibitem[BS1]{BS1} M. Brodmann and P. Schenzel, {\em On varieties of almost minimal degree in small codimension}, J. Algebra 305 , no. 2, 789-–801(2006).

\bibitem[BS2]{BS2} M. Brodmann and P. Schenzel, {\em Arithmetic properties of projective varieties of almost minimal degree}, J. Algebraic Geometry 16 , 347--400(2007).

\bibitem[BE]{BE} D. Buchsbaum and D. Eisenbud, {\em Algebra structures for finite free resolutions, and some structure theorems for ideals of codimension $3$}, Amer. J. Math. 99 , 447--485(1977).

\bibitem[DP]{DP} P. del Pezzo, {\em Sulle superficie ${\rm dell’n}^{mo}$ ordine immerse nello spazio din dimensioni}, Rendiconti del Circolo Matematico di Palermo, volume 1 , 241-–271(1887).

\bibitem[Eis88]{Eis88} D. Eisenbud, {\em Linear sections of determinantal varieties}, Amer. J. Math. 110 , no. 3, 541--575(1988).

\bibitem[Eis05]{Eis05} D. Eisenbud, {\em The Geometry of Syzygies}, no. 229, Springer-Velag New York, (2005).

\bibitem[EH]{EH} D. Eisenbud and J. Harris, {\em On varieties of minimal degree}, Proceedings of Symposia in Pure Mathematics 46 ,
3--13(1987).

\bibitem[Eh]{Eh} S. Ehbauer, {\em Syzygies of points in projective space and applications}, in: Zero-dimensional schemes (Ravello, 1992), Walter de Gruyter \& Co., Berlin, 145--170 (1994)

\bibitem[F]{F} T. Fujita, {\em Classification theories of polarized varieties}, Cambridge University Press, Cambridge, (1990)

\bibitem[GvB]{GvB} H.-C. Graf von Bothmer, {\em Generic syzygy schemes}, J. Pure Appl. Algebra 2008 , no. 3, 867--876(2007).

\bibitem[Gr1]{Gr1} M. Green, {\em The canonical ring of a variety of general type}, Duke Math. J. 49 , 1087--1113(1982).

\bibitem[Gr2]{Gr2} M. Green, {\em Koszul cohomology and the geometry of projective varieties I}, J. Differ. Geom., 19 , 125--171 (1984).

\bibitem[HLMP]{HLMP} K. Han, W. Lee, H. Moon and E. Park, {\em Rank 3 Quadratic Generators of Veronese Embeddings}, Compositio Math. 157 , 2001--2025(2021).

\bibitem[Hoa]{Hoa} Le Tuan Hoa, {\em On minimal free resolutions of projective varieties of degree=codimension+2}, J. Pure Appl. Algebra 87 , 241--250(1993).

\bibitem[MP]{MP} Moon, H., Park, E.,{\em On the Rank Index of Some Quadratic Varieties}, Mediterr. J. Math. 20, 260 (2023). https://doi.org/10.1007/s00009-023-02460-9

\bibitem[NP]{NP} U. Nagel and Y. Pitteloud, {\em On graded Betti numbers and geometrical properties of projective varieties}, Manuscripta Math. 84 , no. 3-4, 291-–314(1994).

\bibitem[P]{P2007} E. Park, {\em Smooth varieties of almost minimal degree}, J. Algebra 314 , 185--208(2007).

\bibitem[P22]{P2022} Park, E. {\em On the Rank of Quadratic Equations for Curves of High Degree}, Mediterr. J. Math. 19, 244 (2022). https://doi.org/10.1007/s00009-022-02170-8

\bibitem[R]{R} M. Reid, {\em Gorenstein in codimension $4$: the general structure theory}, in: Algebraic geometry in East Asia -- Taipei 2011, Adv. Stud. Pure Math. 65, Math. Soc. Japan, Tokyo , 201--227(2015).

\bibitem[SS]{SS} H. Schenck and M. Stillman, {\em High rank linear syzygies on low rank quadrics}, Amer. J. Math. 134 , no. 2, 561--579(2012).

\bibitem[St]{St} I. Stenger, {\em A structure result for Gorenstein algebras of odd codimension}, J. Algebra 589 , 173--187(2022).

\bibitem[Sw]{Sw} H. P. F. Swinnerton-Dyer, {\em An enumeration of all varieties of degree 4}, Amer. J. Math. 95 , 403-–418(1973).

\end{thebibliography}
\end{document}